\documentclass[12pt]{amsart}
\usepackage[hmargin=2.5cm,vmargin=2.5cm]{geometry}
\usepackage{amsmath}
\usepackage{amsfonts}
\usepackage{amssymb}
\usepackage[noadjust]{cite}
\usepackage{enumitem}
\usepackage{setspace}
\usepackage{amsthm}
\usepackage{graphicx}
\usepackage{float}
\usepackage{subcaption}
\usepackage{xcolor}
\usepackage{tikz}

\AtBeginDocument{
   \def\MR#1{}
}

\newcommand{\tikzmark}[2]{\tikz[overlay,remember picture,baseline] \node [anchor=base] (#1) {$#2$};}

\newcommand{\DrawVLine}[3][]{%
  \begin{tikzpicture}[overlay,remember picture]
    \draw[#1] (#2.north) -- (#3.south);
  \end{tikzpicture}
}
\newcommand{\DrawHLine}[3][]{%
  \begin{tikzpicture}[overlay,remember picture]
    \draw[#1] (#2.west) -- (#3.east);
  \end{tikzpicture}
}

\pagecolor{white}

\allowdisplaybreaks

\newcommand{\Z}{\mathbb{Z}}

\newcommand{\R}{\mathbb{R}}

\newcommand{\kb}[1]{\ensuremath{\langle #1 \rangle}}
\newcommand{\asplice}{\ensuremath{\makebox[0.3cm][c]{\raisebox{-0.3ex}{\rotatebox{90}{$\asymp$}}}}}

\newtheorem{thm}{Theorem}
\numberwithin{thm}{section}

\newtheorem{prop}[thm]{Proposition}
\newtheorem{lemma}[thm]{Lemma}
\newtheorem{cor}[thm]{Corollary}
\newtheorem{question}[thm]{Question}

\newtheorem*{namedtheorem}{\theoremname}
\newcommand{\theoremname}{testing}
\newenvironment{named_thm}[1]{\renewcommand{\theoremname}{#1}\begin{namedtheorem}}{\end{namedtheorem}}
\newcommand{\refthm}[1]{Theorem~\ref{thm:#1}}

\theoremstyle{definition}
\newtheorem{defn}[thm]{Definition}
\newtheorem*{nameddef}{\defname}
\newcommand{\defname}{testing}
\newenvironment{named_def}[1]{\renewcommand{\defname}{#1}\begin{nameddef}}{\end{nameddef}}
\newcommand{\refdef}[1]{Definition~\ref{def:#1}}

\theoremstyle{remark}
\newtheorem{rmk}[thm]{Remark}

\begin{document}
\title{The Jones Polynomial from a Goeritz Matrix}
\author{Joe Boninger}
\address{Department of Mathematics, Boston College, Chestnut Hill, MA}
\maketitle

\begin{abstract}
We give an explicit algorithm for calculating the Kauffman bracket of a link diagram from a Goeritz matrix for that link. Further, we show how the Jones polynomial can be recovered from a Goeritz matrix when the corresponding checkerboard surface is orientable, or when more information is known about its Gordon-Litherland form.  In the process we develop a theory of Goeritz matrices for cographic matroids, which extends the bracket polynomial to any symmetric integer matrix. We place this work in the context of links in thickened surfaces.
\end{abstract}

\section{Introduction}
The Jones polynomial of a link can be computed using a link diagram, a closed braid representative, or other, essentially diagrammatic data. It remains an open problem, posed by Atiyah in 1988 \cite{a88}, to give a three-dimensional definition of the polynomial; since there exist links with homeomorphic complements and distinct Jones polynomials, it is not known what topological information determines the Jones polynomial of a link.

In this paper, we give an explicit algorithm for computing the Kauffman bracket of a non-split link from a Goeritz matrix for that link. Goeritz matrices are combinatorial constructions, defined using a checkerboard surface of a link diagram, that have been actively studied  for almost ninety years \cite{g33}. In a recent paper, Traldi \cite{t20}, extending \cite{l90, g13}, proved the link equivalence relation generated by sharing a Goeritz matrix is the same as the equivalence relation generated by Conway mutations; in particular, two links share a Goeritz matrix only if they are related by a sequence of mutations. Thus, our work may be viewed as a partial realization of the Jones polynomial as an invariant of link mutation classes. We define a function
$$
\mu : \{\text{symmetric, integer matrices}\} \to \Z[A^{\pm1}],
$$
and we prove $\mu$ coincides with the Kauffman bracket on matrices which are Goeritz matrices of link diagrams. Further, we show how to recover the full Jones polynomial when the checkerboard surface used to construct the matrix is orientable, or when more information is known about its Gordon-Litherland form.

We prove:

\begin{named_thm}{\protect\refthm{orientable_jones}}
Let $K$ be a knot with checkerboard surface $S$ and associated Goeritz matrix $G = (g_{ij})$. If $S$ is orientable (equivalently, if every diagonal entry of $G$ is even), then the Jones polynomial $J_K(t)$ of $K$ is given by
$$
J_K(t) = \big[ (-A)^{3(\sum_{i \leq j} g_{ij})} \mu[G] \big]_{t^{1/2} = A^{-2}}.
$$
\end{named_thm}

More generally:

\begin{named_thm}{\protect\refthm{full_jones}}
Let $L$ be an oriented link with checkerboard surface $S$ and associated Goeritz matrix $G = (g_{ij})$. Then
$$
J_L(t) = \big[ (-A)^{-3(e(S,L) - \sum_{i \leq j} g_{ij})} \mu[G] \big]_{t^{1/2} = A^{-2}}.
$$
\end{named_thm}
The quantity $e(S,L)$, the signed Euler number of $S$ and $L$, is defined in Section 5 below.

Goeritz matrices have topological significance. A Goeritz matrix of a checkerboard surface represents the Gordon-Litherland form of the surface, and is therefore also a presentation matrix for the first homology group of the branched double cover of the corresponding link \cite[Ch.~9]{l97}. In this context, Theorem \ref{thm:full_jones} says the Jones polynomial of a link can be computed from the Gordon-Litherland form of a sufficiently nice spanning surface, provided we choose the right basis. Our polynomial $\mu$ is not an invariant of quadratic forms---this is impossible, since there exist links which admit isomorphic Gordon-Litherland forms and have distinct Jones polynomials. Nonetheless, we are hopeful our work will lead to a greater topological understanding of the Jones polynomial.

In addition to relating the polynomial $\mu$ to the Jones polynomial and the Kauffman bracket, we investigate its meaning for those symmetric, integer matrices which are not the Goeritz matrix of any classical link diagram. This leads us to consider the construction of a Goeritz matrix from a Tait graph, and we generalize this construction from signed, planar graphs to all colored, cographic matroids---see Sections 2 and 3 for background information and details. After defining Goeritz matrices for cographic matroids, we use them to give an interpretation of the polynomial $\mu$ for any symmetric, integer matrix. We then place the theory we've developed in the context of links in thickened surfaces, with two immediate applications. First, we prove: 

\begin{named_thm}{\protect\refthm{every_matrix}}
Every symmetric, integer matrix is the Goeritz matrix of a checkerboard-colorable link in a thickened surface.
\end{named_thm}

Goeritz matrices for links in thickened surfaces were defined in \cite{ill10}, and studied further in \cite{bck21}. Theorem \ref{thm:every_matrix} extends \cite[Thm.~3.5]{bck21}, which considers the case of knots.

As a second application, for any checkerboard-colorable, non-split, oriented link $L$ in a thickened closed, orientable surface $\Sigma \times I$, we define a set $V$ of two polynomials $\{\nu, \nu'\}$ in one variable $t$---see Definition \ref{def:nu} below. We show the set $V$ is an isotopy invariant of $L$, and that both $\nu$ and $\nu'$ coincide with the Jones polynomial of $L$ when $\Sigma = S^2$. Finally, the work of \cite{ill10, bck21} shows every checkerboard-colorable link in a thickened surface has two determinants, $\det(L) = \{d, d'\}$. We prove:

\begin{named_thm}{\refthm{dets}}
Let $L$ be a checkerboard-colorable, non-split, oriented link in a thickened closed, orientable surface. Then
$$
\det(L) = \{|\nu(-1)|, |\nu'(-1)|\}.
$$
\end{named_thm}

This surprising result precisely generalizes the classical case, where the determinant of a link is equal to the absolute value of its Jones polynomial evaluated at $-1$. Our proof of Theorem \ref{thm:dets} gives a direct proof of the classical case as well, without making use of the Alexander or Kauffman polynomials.

Given a checkerboard-colorable link diagram $D$ in a surface $\Sigma$, the polynomials $V = \{\nu, \nu'\}$ are defined using the two Tait graphs of $D$. Thistlethwaite \cite{t87} defined a Tutte-type polynomial invariant of signed graphs, which we call $\tau$, and we apply $\tau$ to each Tait graph of $D$ to produce the polynomials $\nu$ and $\nu'$. We also reformulate $\tau$ as an invariant of signed matroids in order to relate it to our polynomial $\mu$. Thistlethwaite concluded \cite{t87} by writing, ``I do not know whether this polynomial [$\tau$] has any application in the case that [a graph] is not planar.'' We are pleased to discover such a use for $\tau$.

\subsection{A Polynomial for Symmetric Matrices}

We proceed with an overview of Goeritz matrices and the polynomial $\mu$. Calculation of $\mu$ is easy to automate, and $\mu$ may have practical applications for computing Jones polynomials.

Let $D \subset \R^2$ be a diagram of a non-split link $L \subset S^3$. To define a Goeritz matrix for $D$, we shade one of the checkerboard surfaces $S$ bound by $D$ in the plane. We then label the unshaded regions comprising $\R^2 - S$ by $X_0, X_1, \dots, X_m$. This gives a sign $\sigma(c)$ to each crossing $c \in D$, as shown in Figure \ref{fig:cc}. The \emph{unreduced Goeritz matrix} $\tilde{G} = (\tilde{g}_{ij})$ corresponding to $D$ and $S$ is the $(m + 1)$-by-$(m + 1)$ symmetric matrix defined by
\begin{equation}
\label{eq:goeritz}
\tilde{g}_{ij} = \begin{cases}
-\sum_{\text{$c$ adjacent to $X_i$ and $X_j$}} \sigma(c) & i \neq j \\
-\sum_{k \neq i} g_{ik} & i = j
\end{cases}.
\end{equation}
A \emph{Goeritz matrix} $G = (g_{ij})$ is obtained from $\tilde{G}$ by deleting its first row and column. Thus, $G$ is an $m$-by-$m$ symmetric matrix whose construction depends on a few choices.

\begin{figure}[H]
\centering
\subcaptionbox{Checkerboard sign \\ convention \label{fig:cc}}{
\includegraphics[height=2cm]{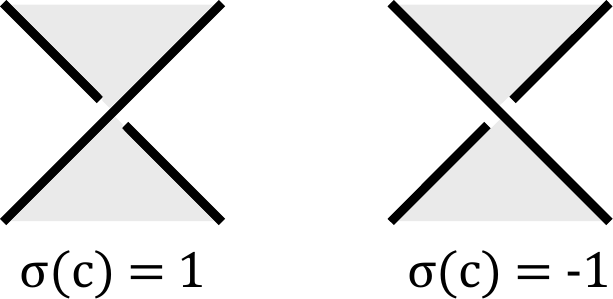}
}
\hspace{1cm}
\subcaptionbox{Writhe sign \\ convention \label{fig:crossing_sign}}{
\includegraphics[height=2cm]{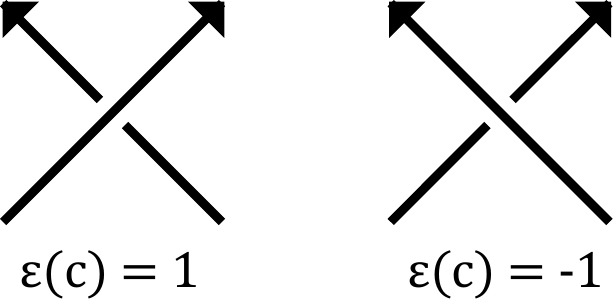}
}
\hspace{1cm}
\subcaptionbox{Crossing types \label{fig:ct}}{
\includegraphics[height=2cm]{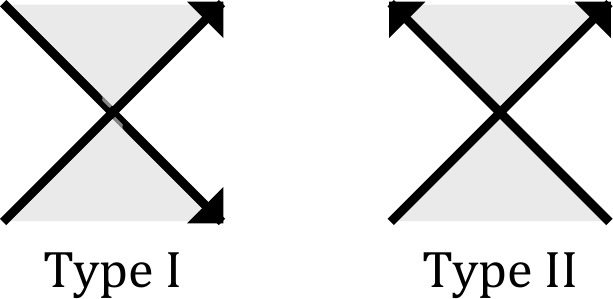}
}
\caption{}
\label{fig:conventions}
\end{figure}

Not every crossing of $D$ is detected by $G$---only those adjacent to two distinct regions of $\R^2 - D$. We call undetected crossings $S$-nugatory.

\begin{defn}
Let $D$ be a link diagram with checkerboard surface $S$. A crossing $c \in D$ is \emph{$S$-nugatory} if there exists a simple closed curve in $\R^2$ which intersects $S$ only at $c$. For any orientation of $D$, we denote the writhe of all $S$-nugatory crossings (following Figure \ref{fig:crossing_sign}) by $w_0(D,S)$.
\end{defn}

Every $S$-nugatory crossing is nugatory, and since a nugatory crossing has the same sign $\varepsilon$ under any orientation, $w_0(S,D)$ does not depend on how we orient $D$.

Next, we define three matrices $G'_{ij}$, $G''_{ij}$, and $G'_i$, which may be obtained from $G$ as follows:

\begin{defn}
\label{def:matrix_1}
Let $g_{ij}$, $i \neq j$, be a fixed entry of $G$.
\begin{enumerate}[label=(\alph*)]
\item Define $G'_{ij} = (g'_{k\ell})$ to be the symmetric matrix with:
\begin{align*}
&g'_{ii} = g_{ii} + g_{ij}, \\ 
&g'_{jj} = g_{jj} + g_{ij}, \\
&g'_{ij}, g'_{ji} = 0,
\end{align*}
and $g'_{k\ell} = g_{k\ell}$ otherwise.
The changes between $G$ and $G'$ are shown below.
$$
\begin{bmatrix}
\ddots & \\
& g_{ii} & \dots & g_{ij} \\
& \vdots & \ddots & \vdots \\
& g_{ji} & \dots & g_{jj} \\
& & & & \ddots \\
\end{bmatrix} \mapsto
\begin{bmatrix}
\ddots & \\
& g_{ii} + g_{ij} & \dots & 0 \\
& \vdots & \ddots & \vdots \\
& 0 & \dots & g_{jj} + g_{ij} \\
& & & & \ddots \\
\end{bmatrix}
$$
\item Define $G''_{ij} = (g''_{k\ell})$ to be the symmetric matrix obtained from $G$ by deleting its $j$th row and column, and assigning the following values:
\begin{align*}
&g''_{ii} = g_{ii} + g_{jj} + 2g_{ij}, \\
&g''_{ik} = g_{ik} + g_{jk} \text{ for all $k \neq i$}, \\
&g''_{ki} = g_{ki} + g_{kj}\text{ for all $k \neq i$}.
\end{align*}
This is shown below, where red lines indicate the removed row and column.
$$
\begin{bmatrix}
\ddots & & & & \vdots \\
& g_{ii} & g_{ik} & \dots & g_{ij} \\
& g_{ki} & \ddots & & g_{kj} \\
& \vdots & & \ddots & \vdots \\
\dots & g_{ji} & g_{jk} & \dots & g_{jj} & \dots \\
& & & & \vdots &  \ddots \\
\end{bmatrix} \mapsto
\begin{bmatrix}
\ddots & & & & \tikzmark{topA}{\vdots} \\
& g_{ii} + g_{jj} + 2g_{ij} & g_{ik} + g_{jk} & \dots & g_{ij} \\
& g_{ki} + g_{kj} & \ddots & & g_{kj} \\
& \vdots & & \ddots & \vdots \\
\tikzmark{topB}{\dots} & g_{ji} & g_{jk} & \dots & g_{jj} & \tikzmark{bottomB}{\dots} \\
& & & & \tikzmark{bottomA}{\vdots} &  \ddots \\
\end{bmatrix}
$$
\DrawVLine[red, thick, opacity=0.8]{topA}{bottomA}
\DrawHLine[red, thick, opacity=0.8]{topB}{bottomB}
\item For a fixed diagonal entry $g_{ii}$ of $G$, let $G'_i$ be the symmetric matrix obtained from $G$ by deleting its $i$th row and column.
\end{enumerate}
\end{defn}

We also require the following polynomials, which arise naturally when applying the Kauffman bracket to twist regions of a link diagram.
\begin{defn}
\label{def:polys}
For an indeterminate $A$, define a set of Laurent polynomials $P_n \in \Z[A^{\pm 1}]$, $n \in \Z$, by
$$
P_n(A) = \sum_{j = 1}^{|n|} (-1)^{j - 1} A^{\text{sgn}(n)(|n| - 4j + 2)}
$$
if $n \neq 0$. Let $P_0 \equiv 0$.
\end{defn}

We can now define our polynomial $\mu$ and state the relevant theorems.
\begin{defn}
\label{def:new_bracket}
Given a symmetric, integer matrix $G = (g_{k\ell})$, define a polynomial $\mu[G] \in \Z[A^{\pm 1}]$ recursively as follows:
\begin{enumerate}[label=(\roman*)]
\item If $G$ is the empty matrix, $\mu[G] = 1$.
\item For any $i \neq j$,
\begin{align*}
\mu[G] = A^{-g_{ij}} \mu[G_{ij}'] + P_{-g_{ij}}(A) \mu[G''_{ij}].
\end{align*}
\item Let $g_{ii}$ be any diagonal entry of $G$ such that $g_{i\ell} = 0$ (and $g_{\ell i} = 0$) for all $\ell \neq i$. Then
$$
\mu[G] = (A^{g_{ii}}(-A^{-2} - A^2) + P_{g_{ii}}(A)) \mu[G_i'].
$$
\end{enumerate}
\end{defn}

\begin{named_thm}{\protect\refthm{well-defined}}
The polynomial $\mu$ is well-defined for any symmetric integer matrix.
\end{named_thm}

\begin{named_thm}{\protect\refthm{bracket}}
Let $L \subset S^3$ be a link with non-split diagram $D$ and checkerboard surface $S$, and let $G$ be a Goeritz matrix associated to $S$. Then
$$
\langle D \rangle = (-A)^{-3w_0(D,S)}\mu[G],
$$
where $\langle D \rangle$ is the Kauffman bracket of $D$.
\end{named_thm}

Assuming the existence of $\mu$, its uniqueness is clear. Let $G$ be any symmetric, integer matrix, and consider applying relation (ii) to an off-diagonal element $g_{ij}$: in the matrix $G'_{ij}$ the corresponding element $g'_{ij}$ is zero, and in the matrix $G''_{ij}$ this element has been removed altogether. Thus, by repeatedly using relation (ii), we can express $\mu(G)$ as a $\Z[A^{\pm 1}]$-linear combination of polynomials $\{\mu(H_k)\}_k$, where each $H_k$ is a diagonal matrix. We can then use relation (iii) to reduce to the case where each $H_k$ is the empty matrix, for which $\mu$ is defined to be $1$.

As an example, consider the matrix
$$
G = \begin{bmatrix}
2 & -1 \\
-1 & 2
\end{bmatrix},
$$
which is a Goeritz matrix for the checkerboard surface of the trefoil shown in Figure \ref{fig:trefoil}. With the method described above, we compute:
\begin{align*}
\mu[G] &= A\mu \begin{bmatrix} 1 & 0 \\ 0 & 1 \end{bmatrix} + A^{-1} \mu[2] \\
&= A\big(A(-A^{-2} - A^2) + A^{-1} \big) \mu[1] + A^{-1} \big(A^2(-A^{-2} - A^2) + (1 - A^{-4}) \big) \\
&= A\big(A(-A^{-2} - A^2) + A^{-1} \big)\big(A(-A^{-2} - A^2) + A^{-1} \big) + A^{-1} \big(A^2(-A^{-2} - A^2) + (1 - A^{-4}) \big) \\
&= A^7 - A^3 - A^{-5}.
\end{align*}
As Theorem \ref{thm:bracket} claims, the polynomial $\mu(G)$ is equal to the Kauffman bracket of the diagram.

\begin{figure}[H]
\centering
\subcaptionbox{A checkerboard surface \\ of the trefoil \label{fig:trefoil}}{
\includegraphics[height=3.5cm]{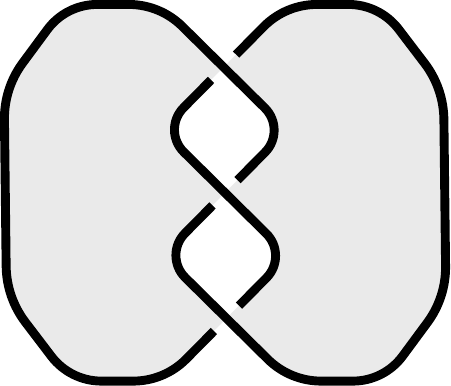}
}
\hspace{1cm}
\subcaptionbox{The Tait graph of \protect{\ref{fig:trefoil}}\label{fig:tait}}{
\includegraphics[height=3.5cm]{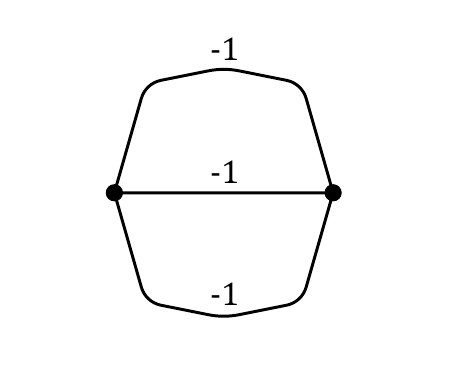}
}
\caption{}
\label{fig:conventions}
\end{figure}

To prove Theorem \ref{thm:well-defined}, and to understand the meaning of $\mu$ for matrices which are not the Goeritz matrices of classical links, we define Goeritz matrices for any signed, cographic matroid in Definitions \ref{def:scb} and \ref{def:goeritz} below. We then show, in Corollary \ref{thm:all_signed}, that every symmetric, integer matrix is a Goeritz matrix of a signed, cographic matroid. We prove the following generalization of Theorem \ref{thm:bracket}:

\begin{named_thm}{\protect\refthm{main_one}}
Let $M = (E, \mathcal{C}, \sigma)$ be a signed, cographic matroid with Goeritz matrix $G$. Let $\iota_+$ (resp. $\iota_-$) be the number of coloops $e$ of $M$ such that $\sigma(e) = 1$ (resp. $\sigma(e) = -1$). Then
$$
\tau[M] = (-A)^{3(\iota_- - \iota_-)}\mu[G].
$$
\end{named_thm}

As above, $\tau$ is Thistlethwaite's invariant of \cite{t87}. Theorem \ref{thm:main_one} implies Theorem \ref{thm:bracket} because, as \cite{t87} shows, if $\Gamma$ is the Tait graph of a link diagram $D \subset \R^2$, then $\tau[\Gamma] = \langle D \rangle$. Theorem \ref{thm:well-defined} follows from Theorem \ref{thm:main_one} and Corollary \ref{thm:all_signed}.

The polynomial $\mu$ raises many interesting questions beyond the scope of this paper. For example, is there a direct proof, without appealing to matroids, that $\mu$ is well-defined? Further:

\begin{question}
What is a closed formula for $\mu[G]$ in terms of the entries of the symmetric matrix $G$?
\end{question}

Even for four-by-four matrices, this question appears hard to answer without computer assistance.

\subsection{Outline}

In Section 2, we review relevant background information. In Section 3 we define Goeritz matrices for cographic matroids, and in Section 4 we use this theory to prove Theorems \ref{thm:main_one}, \ref{thm:well-defined}, and \ref{thm:bracket}. Section 5 shows how the Jones polynomial may be computed from a Goeritz matrix and the Gordon-Litherland form of a checkerboard surface, and we prove Theorems \ref{thm:orientable_jones} and \ref{thm:full_jones}. Finally, Section 6 applies the results of Sections 3 and 4 to links in thickened surfaces. We define the polynomials $V = \{\nu, \nu'\}$ in Definition \ref{def:nu}, and we prove Theorems \ref{thm:every_matrix} and \ref{thm:dets}.

\subsection{Acknowledgements}

We thank Ilya Kofman for his mentorship in this project, and Gabriel Black for pointing out an error in the proof of Theorem \ref{thm:nu}.

\section{Preliminaries}

\subsection{The Kauffman Bracket and the Jones Polynomial}

The Kauffman bracket \cite{k87} of a link diagram $D \subset \R^2$ is a Laurent polynomial $\langle D \rangle \in \Z[A^{\pm1}]$ defined recursively by:
\begin{enumerate}[label=(\roman*)]
\item $\kb{\bigcirc} = 1 $
\item $\kb{\bigcirc \sqcup D} = (-A^{-2}-A^2) \; \kb{D}$
\item $\kb{\includegraphics[height=0.3cm]{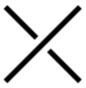}} = A \; \kb{\asplice} + A^{-1}  \kb{\asymp} $
\end{enumerate}
The symbol $\bigcirc$ indicates a simple closed curve, while the arcs in relation (iii) represent three diagrams which are identical outside of a disk, where they look at shown.

The Jones polynomial, an invariant of oriented links, may be obtained from the Kauffman bracket by a normalization. For any oriented link $L$ with diagram $D \subset \R^2$, with writhe $w(D)$, the Jones polynomial $J_L(t)$ is defined by:
\begin{equation}
\label{eq:jones}
J_L(t) = \big[ (-A)^{-3w(D)} \langle D \rangle \big]_{t^{1/2} = A^{-2}}.
\end{equation}

\subsection{Links in Thickened Surfaces}

Let $\Sigma$ be a closed, orientable surface. As in the classical case, a link in a thickened surface is a smooth embedding $\sqcup S^1 \hookrightarrow \Sigma \times I$ of finitely many disjoint copies of $S^1$, considered up to isotopy. In this context, a link diagram $D \subset \Sigma$ is the image of a regular projection
$$
\sqcup S^1 \hookrightarrow \Sigma \times I \to \Sigma \times \{0\} \cong \Sigma,
$$
with over/under crossing information added on $\Sigma$, just as for diagrams in $\R^2$. A link diagram $D \subset \Sigma$ is said to be \emph{checkerboard-colorable} if the connected components of $\Sigma - D$ can each be colored black or white, so that no two regions which abut the same strand of $D$ have the same color. In this case the shaded regions of $\Sigma - D$ form a \emph{checkerboard surface} for $D$.

Im, Lee, and Lee \cite{ill10} extended Goeritz matrices to checkerboard-colorable links in thickened surfaces, as follows:

\begin{defn}[\protect \cite{ill10}]
\label{def:ill}
Let $D \subset \Sigma$ be a checkerboard-colorable link diagram with checkerboard surface $S$, and let $X_0, \dots, X_m$ enumerate the regions of $\Sigma - S$. The \emph{unreduced Goeritz matrix} $\tilde{G} = (\tilde{g}_{ij})$ corresponding to $D$ and $S$ is the $(m + 1)$-by-$(m + 1)$ symmetric matrix defined by
$$
\tilde{g}_{ij} = \begin{cases}
-\sum_{\text{$c$ adjacent to $X_i$ and $X_j$}} \sigma(c) & i \neq j \\
-\sum_{k \neq i} g_{ik} & i = j
\end{cases},
$$
where the sign $\sigma(c)$ of a crossing $c$ is as in Figure \ref{fig:cc}. As before, a \emph{Goeritz matrix} $G = (g_{ij})$ is obtained from $\tilde{G}$ by deleting its first row and column.
\end{defn}

Suppose $D \subset \Sigma$ is a checkerboard-colorable link diagram of a link $L \subset \Sigma \times I$. Let $S$ and $S'$ be the two checkerboard surfaces of $D$, with $G$ and $G'$ Goeritz matrices for $S$ and $S'$ respectively. Im, Lee, and Lee showed combinatorially that the set $\{|\det(G)|, |\det(G')|\}$ is an invariant of $L$, called the \emph{determinant} of $L$. (In the classical case $\Sigma = S^2$, $|\det(G)| = |\det(G')|$.) Their work was generalized and given topological meaning in a recent paper by Boden, Chrisman, and Karimi \cite{bck21}.

\subsection{Tait Graphs and Matroids}

Let $D \subset \Sigma$ be a checkerboard-colorable link diagram with checkerboard surface $S$, where $\Sigma$ is either a closed surface or $\R^2$. The \emph{Tait graph} of $D$ and $S$ is the graph $\Gamma \subset \Sigma$ which assigns a vertex to each region of $S$ in $\Sigma$, and an edge to each crossing $c \in D$ that joins the vertices of its two adjacent regions. The signs $\sigma$ of crossings induce signs on the edges of the Tait graph---see Figure \ref{fig:tait} for the Tait graph of the surface in Figure \ref{fig:trefoil}.

If $D \subset \R^2$ and $D$ is non-split, Goeritz matrices of $D$ may be defined using $\Gamma$ as follows. Let $X_0, X_1, \dots, X_m$ enumerate the regions of $\R^2 - \Gamma$, and let $C_i = \partial X_i \subset \Gamma$. Then a Goeritz matrix $G = (g_{ij})$ of $D$, with $1 \leq i,j  \leq m$, is given by:
$$
g_{ij} = \begin{cases}
-\sum_{e \in C_i \cap C_j} \sigma(e) & i \neq j \\
\sum_{e \in C_i, \text{ $e$ not a bridge}} \sigma(e) & i = j
\end{cases}.
$$
Here $e$ is an edge of $\Gamma$ and $\sigma(e)$ its sign. Equivalence with our previous definition is clear.

It will be advantageous for us to work not just with graphs but with \emph{matroids}, which share many of their properties. A \emph{matroid} is a pair $M = (E, \mathcal{C})$, where $E$ is a finite set, called the \emph{ground set} or \emph{set of points}, and $\mathcal{C}$ a family of subsets of $E$, called \emph{circuits}, such that:
\begin{enumerate}[label=(\roman*)]
\item $\varnothing \notin \mathcal{C}$
\item If $C \in \mathcal{C}$, no proper subset of $C$ is in $\mathcal{C}$.
\item If $C, C' \in \mathcal{C}$ are distinct circuits with $e \in C \cap C'$, then $(C \cup C') - \{e\}$ contains a circuit.
\end{enumerate}
A point $e \in E$ is called a \emph{loop} if $\{e\} \in \mathcal{C}$, and a \emph{coloop} if $e$ is not contained in any circuit. A maximal subset of $E$ which does not contain a circuit is called a \emph{basis}. Every matroid $M$ has a dual matroid $M^*$ over the same ground set, with $(M^*)^* = M$, defined by the condition that a set of points is a basis of $M^*$ if and only if its complement is a basis of $M$.

Any undirected graph $\Gamma = (E,V)$ gives rise to two dual matroids. The \emph{cycle matroid} of $\Gamma$, $M(\Gamma)$, is defined by letting circuits be simple cycles of $\Gamma$. Loops of $\Gamma$ are loops of $M(\Gamma)$, bridges of $\Gamma$ are coloops of $M(\Gamma)$, and spanning forests of $\Gamma$ are bases. A matroid which is isomorphic to the cycle matroid of some graph is called \emph{graphic}, and we will sometimes conflate a graph with its cycle matroid. Alternatively, the dual matroid $M^*(\Gamma)$ is called the \emph{bond matroid} of $\Gamma$, and is denoted $B(\Gamma)$. Circuits of $B(\Gamma)$, called \emph{bonds}, are minimal collections of edges, which, when removed from $\Gamma$, increase its number of connected components. A matroid isomorphic to the bond matroid of some graph is called \emph{cographic}. A graphic matroid $M(\Gamma)$ is cographic if and only if the underlying graph $\Gamma$ is planar \cite{w32}. In this case $M(\Gamma) = B(\Gamma^*)$, where $\Gamma^*$ is the planar dual of $\Gamma$.

Finally, by a \emph{colored matroid}, we indicate a matroid equipped with a function $\sigma : E \to \Z$. A colored matroid is \emph{signed} if $\text{Im}(\sigma) \subset \{-1, 1\}$. We adopt the convention that if $M$ is a matroid with coloring $\sigma$, its dual matroid $M^*$ has coloring $-\sigma$. In particular, if $\Gamma$ is a colored graph, its cycle matroid $M(\Gamma)$ inherits the same coloring function $\sigma$, while the bond matroid $B(\Gamma)$ inherits $-\sigma$. This matches the case of dual Tait graphs of a link diagram, where dual edges have opposite signs.

\subsection{Cographic Matroids and 2-Bases}

For a matroid $M$ with ground set $E$, let $[E]$ denote the $\Z/2$-vector space generated by $E$. If $A \subset E$, let $\bar{A} \in [E]$ be the sum $\bar{A} = \sum_{a \in A} a$. For any two sets $A, B \subset E$, $\bar{A} + \bar{B} = \overline{A \Delta B}$, where $A \Delta B$ is the symmetric difference $A \cup B - A \cap B$. The following idea of a $2$-basis was introduced by Mac Lane in the context of planar graphs.

\begin{defn}[\protect \cite{m37}]
\label{def:scb}
Let $M = (E, \mathcal{C})$ be a matroid, and define the \emph{circuit space} of $M$ to be the subspace of $[E]$ generated by the set $\{\bar{C} \mid C \in \mathcal{C} \}$. A \emph{2-basis} is a set $\mathcal{B} = \{C_1, \dots, C_m\} \subset \mathcal{C}$, with $\{\bar{C}_1, \dots, \bar{C}_m\}$ a basis for the circuit space of $M$, such that for any three distinct $C_i, C_j, C_k \in \mathcal{B}$, $C_i \cap C_j \cap C_k = \varnothing$.
\end{defn}

A matroid is \emph{binary} if the symmetric difference of any two circuits is a disjoint union of circuits. Welsh showed cographic matroids are characterized among binary matroids by the existence of a 2-basis.

\begin{thm}[\protect \cite{w69}]
\label{thm:star_binary}
A binary matroid admits a 2-basis if and only if it is cographic.
\end{thm}

Let $\Gamma$ be a graph with vertex set $V = \{v_0, v_1, \dots, v_m\}$ and edge set $E$, and define a set of subgraphs $A_i$ by
$$
A_i = \{e \in E \mid e \text{ is not a loop of $\Gamma$}, v_i \text{ is an endpoint of } e\}.
$$
One may check that the set $\mathcal{A} = \{A_1, \dots, A_m\}$ is a 2-basis of the bond matroid $B(\Gamma)$. Every 2-basis of a cographic matroid may be assumed to have this form.

\section{Goeritz Matrices of Cographic Matroids}

In Section 2.3, we defined a Goeritz matrix of a link diagram using a Tait graph. Examining Definition \ref{def:scb}, we see the set of cycles used in the construction forms a 2-basis of the graph. This insight, along with Theorem \ref{thm:star_binary}, leads us to define Goeritz matrices for any colored, cographic matroid.

\begin{defn}
\label{def:goeritz}
Let $M = (E, \mathcal{C}, \sigma)$ be a cographic matroid with coloring $\sigma : E \to \Z$, and let $\mathcal{B} = \{C_1, \dots, C_m\}$ be a 2-basis of $M$. Define a symmetric matrix $G = (g_{ij})$ by:
$$
g_{ij} = \begin{cases}
-\sum_{e \in C_i \cap C_j} \sigma(e) & i \neq j \\
\sum_{e \in C_i} \sigma(e) & i = j
\end{cases}.
$$
Then $G$ is a \emph{Goeritz matrix} of $M$.
\end{defn}

If $\Gamma$ is a Tait graph of a link diagram $D \subset \R^2$, any Goeritz matrix of $D$ (in the classical sense) is a Goeritz matrix of $M(\Gamma)$ by the above definition. The relationship between Definition \ref{def:goeritz} and Definition \ref{def:ill} is more subtle, and will be explained in Section 6.2.

The following, fundamental fact motivates our extension of Goeritz matrices to cographic matroids.

\begin{prop}
\label{thm:fund_goeritz}
Let $G$ be a symmetric, integer matrix. Then there is a colored, cographic matroid $M$ such that $G$ is a Goeritz matrix of $M$.
\end{prop}

\begin{proof}
Suppose $G = (g_{k\ell})$ has dimension $m$. Let $\Gamma = (E,V)$ be the fully connected, simple graph on $m + 1$ vertices $V = \{v_0, v_1, \dots, v_m\}$, and for $0 \leq i < j \leq m$, let $e_{ij}$ be the unique edge connecting $v_i$ and $v_j$. Define a coloring $\sigma : E \to \Z$ on $\Gamma$ by
$$
\sigma(e_{ij}) = \begin{cases}
g_{ij} & i \neq 0 \\
-\sum_{k = 0}^m g_{k j} & i = 0
\end{cases}.
$$
Let $M = B(\Gamma)$. As in Section 2.4, the set $\{C_1, \dots, C_m\}$, where
$$
C_i = \{e \in E \mid v_i \text{ is an endpoint of } e\} = \{e_{k\ell} \in E \mid k = i \text{ or } \ell = i \},
$$
is a 2-basis of $M$. Recalling our convention that the coloring function of $B(\Gamma)$ is $-\sigma$, we calculate, for fixed $i < j \in \{1, \dots, m\}$,
$$
-\sum_{e \in C_i \cap C_j} -\sigma(e) = \sigma(e_{ij}) = g_{ij},
$$
and
$$
\sum_{e \in C_i} -\sigma(e) = \sum_{k = i \text{ or } \ell = i} -\sigma(e_{k \ell}) = \sum_{k = 0}^m g_{ik} - \sum_{k \neq i} g_{ik} = g_{ii}.
$$
Thus, the Goeritz matrix of $M$ is precisely $G$.
\end{proof}

In contrast to Proposition \ref{thm:fund_goeritz}, not every symmetric, integer matrix is a Goeritz matrix of a planar graph. This can be proven, for example, using the Four-Color Theorem, which implies that any five-by-five Goeritz matrix of a planar graph must contain a $0$.

Recall that a signed matroid is a colored matroid $M = (E, \mathcal{C}, \sigma)$ with Im$(\sigma) \subset \{-1, 1\}$.
\begin{cor}
\label{thm:all_signed}
Let $G$ be a symmetric, integer matrix. Then there is a signed, cographic matroid $M$ such that $G$ is a Goeritz matrix of $M$.
\end{cor}

\begin{proof}
Given $G$, let $\Gamma$ be the colored graph constructed in the proof of Proposition \ref{thm:fund_goeritz}. We will change $\Gamma$ to produce a signed graph $\Gamma'$ so that $B(\Gamma)$ and $B(\Gamma')$ have the same Goeritz matrices. If $e \in E$ connects vertices $v_i$ and $v_j$ in $\Gamma$, with $\sigma(e) = n > 0$, we replace $e$ with $n$ edges connecting $v_i$ and $v_j$, all with sign $1$. If $\sigma(e) < 0$, we perform the same operation with negatively signed edges. Finally, if $\sigma(e) = 0$, we replace $e$ with one positive and one negative edge, or delete it from the graph entirely if doing so does not disconnect the graph. The result of these operations is $\Gamma'$.
\end{proof}

To prove our polynomial $\mu$ of Definition \ref{def:new_bracket} is well-defined, we will need to understand how the Goeritz matrix of a cographic matroid changes under certain contractions and deletions. Given a matroid $M = (E, \mathcal{C})$, let $A \subset E$. The \emph{deletion of $M$ with respect to $A$}, denoted $M \setminus A$, is the matroid with ground set $E - A$ and circuits
$$
\mathcal{C}(M \setminus A) = \{C \subset E - A \mid C \in \mathcal{C}(M)\}.
$$
The \emph{contraction of $M$ by $A$}, denoted $M/A$, is defined by $M/A = (M^* \setminus A)^*$. Equivalently, if
$$
\tilde{\mathcal{C}} =  \{C \subset E - A \mid \text{there exists $B \subset A$ such that }C \cup B \in \mathcal{C}(M) \},
$$
then $\mathcal{C}(M/A)$ is the set of minimal elements of $\tilde{\mathcal{C}}$. If $e$ is a loop or coloop of $M$, $M/ \{e\} = M \setminus \{e\}$.

If $M$ is graphic, contraction and deletion operations on $M$ correspond to the familiar contraction and deletion operations on an underlying graph. If $M = (E, \mathcal{C}, \sigma)$ is colored and $A \subset E$, restricting $\sigma$ to $E - A$ induces a coloring on $M \setminus A$ and $M/A$.

In the lemmas that follow, let $M = (E, \mathcal{C}, \sigma)$ be a colored cographic matroid, and $\mathcal{B} = \{C_1, \dots, C_m\}$ a 2-basis of $M$ with Goeritz matrix $G = (g_{k\ell})$. The matrices $G'_{ij}$, $G''_{ij}$, and $G'_i$ are as in Definition \ref{def:matrix_1}.

\begin{lemma}
\label{thm:twist_change_1}
For fixed, distinct $C_i, C_j \in \mathcal{B}$ with $C_i \cap C_j \neq \varnothing$, let $E_{ij} = C_i \cap C_j$.
\begin{enumerate}[label=(\alph*)]
\item $M / E_{ij}$ has an $m$-by-$m$ Goeritz matrix equal to $G'_{ij}$.
\item $M \setminus E_{ij}$ has an $(m - 1)$-by-$(m - 1)$ Goeritz matrix equal to $G''_{ij}$.
\end{enumerate}
\end{lemma}

\begin{proof}
The lemma is easy to verify using the explicit $2$-basis given in Section 2.4, and the fact that contraction and deletion are dual operations. To see how deletion (resp. contraction) affects the bond matroid of a graph $\Gamma$, we can perform the corresponding contraction (resp. deletion) on $\Gamma$ and examine the bond matroid of the resulting graph---see Figure \ref{fig:con_del}. In part (a), a 2-basis for $M/E_{ij}$ is obtained from $\mathcal{B}$ by restricting each cycle to $E - E_{ij}$---this leaves every cycle unchanged except $C_i$ and $C_j$, and affects the stated changes to the Goeritz matrix.

For part (b), a 2-basis for $M \setminus E_{ij}$ is given by removing $C_j$ from $\mathcal{B}$ and replacing $C_i$ with $C_i \Delta C_j$. Again, it is not difficult to check the Goeritz matrix.
\end{proof}

\begin{figure}[H]
\centering
\subcaptionbox{Contraction}{
\includegraphics[height=2cm]{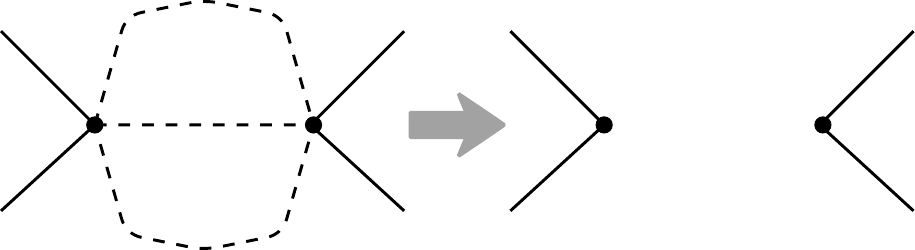}
}
\hspace{1cm}
\subcaptionbox{Deletion}{
\includegraphics[height=2cm]{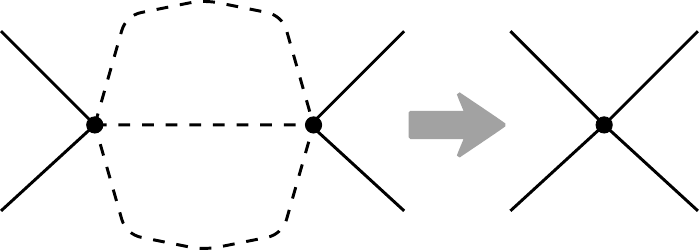}
}
\caption{Contraction and deletion in a bond matroid}
\label{fig:con_del}
\end{figure}

\begin{lemma}
\label{thm:twist_change_2}
For fixed $i$, suppose $C_i$ satisfies $C_i \cap C_j = \varnothing$ for all $j \neq i$. Then $M \setminus C_i$ has an $(m - 1)$-by-$(m - 1)$ Goeritz matrix equal to $G'_i$.
\end{lemma}
\begin{proof}
The proof technique is the same as that of Lemma \ref{thm:twist_change_1}. It is straightforward to check that $\mathcal{B} - \{C_i\}$ is a 2-basis for $M \setminus C_i$, with the desired Goeritz matrix.
\end{proof}

\section{The Polynomial $\mu$}

\subsection{Thistlethwaite's Polynomial $\tau$}

In \cite{t87}, Thistlethwaite introduced a one-variable Laurent polynomial $\tau$ associated to signed matroids. The polynomial $\tau$ may be defined recursively by the relations:
\begin{enumerate}[label=(\roman*)]
\item If $M$ is the empty matroid, $\tau[M] = 1$.
\item Let $e \in E$ be such that $e$ is neither a loop nor coloop. If $\sigma(e)  = +1$, then
$$
\tau[M] = A\tau[M / \{e\}] + A^{-1}\tau[M\setminus\{e\}].
$$
If $\sigma(e) = -1$, then
$$
\tau[M] = A^{-1}\tau[M / \{e\}] + A\tau[M\setminus\{e\}].
$$
\item If $e \in E$ is a loop with $\sigma(e) = -1$, or a coloop with $\sigma(e) = +1$, then
$$
\tau[M] = (-A)^{-3}\tau[M \setminus \{e\}].
$$
If $e \in E$ is a loop with $\sigma(e) = +1$, or a coloop with $\sigma(e) = -1$, then
$$
\tau[M] = (-A)^3\tau[M \setminus \{e\}].
$$
\end{enumerate}

Thistlethwaite's polynomial was subsequently generalized by Kauffman \cite{k89}, and further in \cite{t89, br99}. Thistlethwaite proved that if $D \subset \R^2$ is a non-split link diagram, and $\Gamma$ a Tait graph of $D$, then
\begin{equation}
\label{eq:mt}
\langle D \rangle = \tau[\Gamma].
\end{equation}

\begin{rmk}
Thistlethwaite defined $\tau$ as a signed graph invariant rather than a matroid invariant, and included the stipulation that if $\Gamma$ is a graph with $k$ connected components $\Gamma_1, \Gamma_2, \dots, \Gamma_k$, then
$$
\tau[\Gamma] = (-A^{-2} - A^2)^{k - 1} \tau[\Gamma_1] \tau[\Gamma_2] \cdots \tau[\Gamma_k].
$$
With this relation, (\ref{eq:mt}) holds even for split link diagrams. This requirement does not make sense for matroids, however, since graphs with different numbers of components may have isomorphic cycle matroids.
\end{rmk}

\begin{rmk}
\label{rmk:duality}
With our convention that dual matroids $M$ and $M^*$ have opposite-signed points, the polynomial $\tau$ satisfies $\tau[M] = \tau[M^*]$.
\end{rmk}

To prove Theorem \ref{thm:main_one}, we need two technical lemmas. As in the previous section, let $M = (E, \mathcal{C}, \sigma)$ be a signed cographic matroid and $\mathcal{B} = \{C_1, \dots, C_m\}$ a 2-basis of $M$. For $A \subset E$, define
$$
\sigma(A) = \sum_{e \in A} \sigma(e).
$$

Additionally, recall from Definition \ref{def:polys} that, if $n \neq 0$,
$$
P_n(A) = \sum_{k = 1}^{|n|} (-1)^{k - 1} A^{\text{sgn}(n)(|n| - 4k+ 2)},
$$
and $P_0(A) \equiv 0$.

\begin{lemma}
\label{thm:tech_1}
As in Lemma \ref{thm:twist_change_1}, let $C_i, C_j \in \mathcal{B}$ be distinct, with $E_{ij} = C_i \cap C_j \neq \varnothing$. Then
\begin{equation}
\label{eq:tech}
\tau[M] = A^n \tau[M/E_{ij}] + P_n(A)\tau[M \setminus E_{ij}].
\end{equation}
\end{lemma}

\begin{proof}
Observe that, as a proper subset of the circuit $C_i$, no element of $E_{ij}$ is a loop or coloop. We'll prove the result for $n \geq 0$ by induction---the case of negative $n$ is analogous.

If $n = 0$, (\ref{eq:tech}) reduces to $\tau[M] = \tau[M/E_{ij}]$. Because $\sigma(E_{ij}) = 0$ and $E_{ij} \neq \varnothing$, $E_{ij}$ must contain two elements $e_+$ and $e_-$ such that $\sigma(e_\pm) = \pm 1$. Observe that $e_-$ is a coloop in $M \setminus \{e_+\}$. Indeed, as in Lemma \ref{thm:twist_change_1}, a basis for the circuit space of $M \setminus \{e_+\}$ can be obtained from $\mathcal{B}$ by removing $C_j$ and replacing $C_i$ with the symmetric difference $C_i \Delta C_j$. Since $e_-$ is not contained in any element of this new basis, $e_-$ is not contained in any circuit of $M \setminus \{e_+\}$. Similarly, $e_+$ is a coloop of $M \setminus \{e_-\}$. Finally, note that $e_-$ is \emph{not} a loop or coloop of $M /\{e_+\}$, since it is a proper subset of the circuit $C_i - \{e_+\}$. Applying these facts and the defining properties of $\tau$, we compute
\begin{align*}
\tau[M] &= A \tau[M /\{e_+\}] + A^{-1} \tau[M \setminus \{e_+\}] \\
&= \tau[(M/\{e_+\})/\{e_-\}] + A^2 \tau[(M/\{e_+\}) \setminus \{e_-\}] - A^2 \tau[(M \setminus \{e_+\}) \setminus \{e_-\}] \\
&= \tau[M/\{e_+, e_-\}] + A^2 \tau[M \setminus \{e_+, e_-\}] - A^2 \tau[M \setminus \{e_+, e_-\}] \\
&= \tau[M/\{e_+, e_-\}].
\end{align*}
The third equation uses basic properties of deletion and contraction---see \cite[Ch.~3]{o11}. By repeatedly deleting pairs of elements with opposite signs from $E_{ij}$, the above calculation implies $\tau[M] = \tau[M/E_{ij}]$ when $n = 0$.

Next, if $\sigma(E_{ij}) = n > 0$, then there is at least one $e \in E_{ij}$ such that $\sigma(e) = 1$. It follows that
$$
\tau[M] = A \tau[M / \{e\}] + A^{-1} \tau[M \setminus \{e\}].
$$
If $n = 1$, the above equation is all we needed to show. If not, similar to Lemma \ref{thm:twist_change_1}, the 2-basis $\mathcal{B}$ induces a 2-basis $\mathcal{B}' = \{C_1', \dots, C'_m\}$ of $M / \{e\}$, where $C_k'$ is the restriction of $C_k$ to $E - \{e\}$. Let $E'_{ij} = C'_i \cap C'_j = E_{ij} - \{e\}$; then $\sigma(E'_{ij}) = n - 1$. By the induction hypothesis, the above equation becomes
\begin{align}
\tau[M] &= A^n \tau[(M / \{e\}) / E'_{ij} ] + A P_{n - 1}(A) \tau[(M / \{e\}) \setminus E'_{ij}] + A^{-1} \tau[M \setminus \{e\}] \nonumber \\
&= A^n \tau[M/E_{ij}] + A P_{n - 1}(A) \tau[(M \setminus E'_{ij}) \setminus \{e\}] + A^{-1} \tau[M \setminus \{e\}] \nonumber \\
&= A^n \tau[M/E_{ij}] + A P_{n - 1}(A) \tau[E \setminus E_{ij}] + A^{-1} \tau[M \setminus \{e\}].
\label{eq:stuff}
\end{align}
The second equation follows again from properties of deletion and contraction, using the fact that $e$ is a coloop of $M \setminus E'_{ij}$. Additionally, as in the base case, every element of $E'_{ij}$ is a coloop of $M \setminus \{e\}$. Using the relation (ii) defining $\tau$,
$$
\tau[M \setminus \{e\}] = (-A^{-3})^{n - 1} \tau[(M \setminus \{e\}) \setminus E'_{ij}] = (-1)^{n - 1}A^{-3n + 3} \tau[M \setminus E_{ij}].
$$
Now (\ref{eq:stuff}) becomes
$$
\tau[M] = A^n \tau[M/E_{ij}] + \big(A P_{n - 1}(A) + (-1)^{n - 1} A^{-3n + 2}\big) \tau[M \setminus E_{ij}],
$$
and the proof is completed by showing inductively that $P_n(A) = A P_{n - 1}(A) + (-1)^{n - 1}A^{-3n + 2}$.
\end{proof}

\begin{lemma}
\label{thm:tech_2}
As in Lemma \ref{thm:twist_change_2}, suppose $C_i \in \mathcal{B}$ satisfies $C_i \cap C_j = \varnothing$ for all $j \neq i$, and $\sigma(C_i) = n$. Then
$$
\tau[M] = \big(A^n(-A^{-2} - A^2) + P_n(A)\big) \tau[M \setminus C_i].
$$
\end{lemma}

\begin{proof}
Once again, the proof is by induction on $n$, and we will prove only the nonnegative case. If $n = 0$ and $|C_i| > 2$, we can remove a pair $e_-, e_+ \in C_i$ with $\sigma(e_\pm) = \pm 1$ just as in the proof of Lemma \ref{thm:tech_1}. Thus, for $n = 0$, it suffices to consider the case where $|C_i| = 2$. If $|C_i| = 2$ then $C_i = \{e_-, e_+\}$ for some $e_-, e_+ \in E$, with $\sigma(e_\pm) = \pm1$ as before. This scenario is similar to the previous, but we observe that $e_\pm$ is coloop in $M \setminus \{e_\mp\}$ and $e_\pm$ is a loop in $M / \{e_\mp\}$. We then compute:
\begin{align*}
\tau[M] &= A \tau[M / \{e_+\}] + A^{-1} \tau[M \setminus \{e_+\}] \\
&= -A^{-2} \tau[(M / \{e_+\}) \setminus \{e_-\}] - A^2 \tau[(M \setminus \{e_+\}) \setminus \{e_-\}] \\
&= (-A^{-2} - A^2) \tau[M \setminus C_i],
\end{align*}
which is what we needed to show.

If $n = 1$, by removing pairs of opposite-signed elements, we can reduce to the case where $C_i$ consists of a single element $e$ with $\sigma(e) = 1$. In this case $\sigma(e)$ is a loop, so
$$
\tau[M] = -A^3 \tau[M \setminus C_i].
$$
Noting $P_1(A) = A^{-1}$, we see this is the desired result.

For general $n > 1$, we may choose $e \in C_i$ with $\sigma(e) = 1$. Since $e$ is neither a loop nor coloop,
\begin{equation}
\label{eq:tech_2}
\tau[M] = A\tau[M/\{e\}] + A^{-1} \tau[M \setminus \{e\}].
\end{equation}
Like the proof of Lemma \ref{thm:tech_1}, we can apply the induction hypothesis to the loop $C_i - \{e\}$ in $M/\{e\}$. Additionally, as before, we observe that every element of $C_i - \{e\}$ is a coloop in $M \setminus \{e\}$. With these facts in mind, (\ref{eq:tech_2}) becomes
\begin{align*}
\tau[M] &= A\big(A^{n - 1}(-A^{-2} - A^2) + P_{n - 1}(A)\big)\tau[M \setminus C_i] + A^{-1}(-A)^{-3(n - 1)}\tau[M \setminus C_i] \\
&= \big(A^n(-A^{-2} - A^2) + P_n(A)\big) \tau[M \setminus C_i].
\end{align*}
The identity $P_n(A) = A P_{n - 1}(A) + (-1)^{n - 1}A^{-3n + 2}$ is the same identity appearing in the proof of Lemma \ref{thm:tech_1}.
\end{proof}

\subsection{Proofs of Main Theorems}

In this section, we prove Theorems \ref{thm:main_one}, \ref{thm:well-defined}, and \ref{thm:bracket}. We recall the definition of $\mu$:

\begin{named_def}{\refdef{new_bracket}}
Given a symmetric, integer matrix $G = (g_{k\ell})$, define a polynomial $\mu[G] \in \Z[A^{\pm 1}]$ recursively as follows:
\begin{enumerate}[label=(\roman*)]
\item If $G$ is the empty matrix, $\mu[G] = 1$.
\item For any $i \neq j$,
\begin{align*}
\mu[G] = A^{-g_{ij}} \mu[G_{ij}'] + P_{-g_{ij}}(A) \mu[G''_{ij}].
\end{align*}
\item Let $g_{ii}$ be any diagonal entry of $G$ such that $g_{i\ell} = 0$ (and $g_{\ell i} = 0$) for all $\ell \neq i$. Then
$$
\mu[G] = (A^{g_{ii}}(-A^{-2} - A^2) + P_{g_{ii}}(A)) \mu[G_i'].
$$
\end{enumerate}
\end{named_def}

\begin{thm}
\label{thm:main_one}
Let $M = (E, \mathcal{C}, \sigma)$ be a signed, cographic matroid with Goeritz matrix $G$. Let $\iota_+$ (resp. $\iota_-$) be the number of coloops $e$ of $M$ such that $\sigma(e) = 1$ (resp. $\sigma(e) = -1$). Then
$$
\tau[M] = (-A)^{3(\iota_- - \iota_+)}\mu[G].
$$
\end{thm}

\begin{proof}
The proof is by induction on the number of elements of $E$ which are not coloops. If every element of $E$ is a coloop, $G$ is the empty matrix and $\mu[G]$ is defined to be $1$. Conversely, by the relation (ii) defining $\tau$, $\tau[M] = (-A)^{3(\iota_- - \iota_+)}$. Thus, the theorem holds.

For the general case, let $\mathcal{B}$ be the 2-basis of $M$ corresponding to $G$. Suppose that, for distinct $C_i, C_j \in \mathcal{B}$, $E_{ij} = C_i \cap C_j \neq \varnothing$. Then, by Lemma \ref{thm:tech_1} and the definition of $G$,
$$
\tau[M] = A^{-g_{ij}} \tau[M/E_{ij}] + P_{-g_{ij}}(A)\tau[M \setminus E_{ij}].
$$
By Lemma \ref{thm:twist_change_1}, $G'_{ij}$ is a Goeritz matrix for $M/E_{ij}$ and $G''_{ij}$ is a Goeritz matrix for $M \setminus E_{ij}$. From the induction hypothesis and the relation (ii) defining $\mu$, it follows that
$$
\tau[M] = (-A)^{3(\iota_- - \iota_+)}\big(A^{-g_{ij}} \mu[G'_{ij}] + P_{-g_{ij}}(A)\mu[G_{ij}'']\big) = (-A)^{3(\iota_- - \iota_+)}\mu[G]
$$
as desired.

Finally, if no such $C_i$ and $C_j$ exist, we may choose any cycle $C_i \in \mathcal{B}$. By Lemma \ref{thm:tech_2},
$$
\tau[M] = \big(A^{g_{ii}}(-A^{-2} - A^2) + P_{g_{ii}}(A)\big) \tau[M \setminus C_i].
$$
By Lemma \ref{thm:twist_change_2}, $G'_i$ is a Goeritz matrix for $M \setminus C_i$. The induction hypothesis and the relation (iii) defining $\mu$ imply
$$
\tau[M] = (-A)^{3(\iota_- - \iota_+)}\big(A^{g_{ii}}(-A^{-2} - A^2) + P_{g_{ii}}(A)\big) \mu[G'_i] = (-A)^{3(\iota_- - \iota_+)} \mu[G],
$$
completing the proof.
\end{proof}

We can now prove Theorem \ref{thm:well-defined}.
\begin{thm}
\label{thm:well-defined}
The polynomial $\mu$ is well-defined for any symmetric, integer matrix.
\end{thm}

\begin{proof}
We have shown $\mu$ is well-defined for any matrix $G$ such that $G$ is a Goeritz matrix of a signed, cographic matroid. By Corollary \ref{thm:all_signed}, all symmetric, integer matrices satisfy this condition.
\end{proof}

Theorem \ref{thm:bracket} also follows easily from Theorem \ref{thm:main_one}.

\begin{thm}
\label{thm:bracket}
Let $L \subset S^3$ be a link with non-split diagram $D$ and checkerboard surface $S$, and let $G$ be a Goeritz matrix associated to $S$. Then
$$
\langle D \rangle = (-A)^{-3w_0(D,S)}\mu[G].
$$
\end{thm}

\begin{proof}
Let $\Gamma$ be the Tait graph of $S$. From (\ref{eq:mt}) and Theorem \ref{thm:main_one},
$$
\langle D \rangle = \tau[\Gamma] = (-A)^{3(\iota_- - \iota_+)}\mu[G].
$$
One may check that $-w_0(D,S) = \iota_- - \iota_+$.
\end{proof}

\section{Recovering the Jones Polynomial}

Here, we discuss what information is needed to recover the full Jones polynomial from a Goeritz matrix. Let $L$ be a link with diagram $D$, checkerboard surface $S$, and Goeritz matrix $G$---as Section 2.1 discusses, the Jones polynomial of $L$ can be obtained from $\langle D \rangle$ if we know the writhe of $D$. When $S$ is orientable, $w(D)$ can be read directly from $G$. Further, $G$ encodes the orientability of $S$.

\begin{prop}
\label{thm:o_from_g}
Let $S$ be a checkerboard surface with Goeritz matrix $G = (g_{k\ell})$. Then $S$ is orientable if and only if $g_{ii}$ is even for all $i$.
\end{prop}

\begin{proof}
Le $\Gamma$ be the Tait graph of $S$. Each diagonal element $g_{ii}$ of $G$ corresponds to a simple cycle $C_i$ of $\Gamma$, which represents a class $[C_i] \in H_1(S; \Z/2)$. The number $g_{ii}$ is the number of half-twists in a tubular neighborhood $N(C_i) \subset S$ of $C_i$, counted with sign. If $g_{ii}$ is even then $N(C_i)$ is an annulus, and if $g_{ii}$ is odd then $N(C_i)$ is a M\"obius strip. Since the set of all $C_i$ form a basis of $H_1(S; \Z/2)$, $S$ is orientable if and only if each $g_{ii}$ is even.
\end{proof}

An orientation on $S$ induces an orientation on the components of $D$ such that every crossing of $D$ is of type I, as shown in Figure \ref{fig:ct}. Further, for a type I crossing $c$, $\varepsilon(c) = -\sigma(c)$. This observation and a direct computation lead to the following formula.
\begin{lemma}
\label{o_writhe}
Let $D$ be a link diagram with orientation induced by an oriented checkerboard surface $S$, and let $G$ be a corresponding Goeritz matrix. Then
$$
w(D) - w_0(D,S) = -\sum_{i \leq j} g_{ij}.
$$
\end{lemma}

If $D$ is a knot diagram and $S$ is orientable, an orientation of $D$ is always compatible with an orientation of $S$. With this in mind, combining Proposition \ref{thm:o_from_g}, Lemma \ref{o_writhe}, and Theorem \ref{thm:bracket}, we have:
\begin{thm}
\label{thm:orientable_jones}
Let $K$ be a knot with checkerboard surface $S$ and associated Goeritz matrix $G$. If $S$ is orientable (equivalently, if every diagonal entry of $G$ is even), then
$$
J_K(t) = \big[ (-A)^{3(\sum_{i \leq j} g_{ij})} \mu[G] \big]_{t^{1/2} = A^{-2}}.
$$
\end{thm}

The same result holds for any oriented link $L$, provided the orientation of $L$ is compatible with an orientation of $S$. We recall as well that every link admits an orientable checkerboard surface.

It is also possible to compute the full Jones polynomial from $G$ if $S$ is not orientable, provided we have more information about the \emph{Gordon-Litherland form} of $S$. The Gordon-Litherland form \cite{gl78} is a bilinear, symmetric pairing $\mathcal{G}_S : H_1(S) \times H_1(S) \to \Z$, defined as follows. Let $N(S) \subset S^3$ be the unit normal bundle of $S$, viewed as a subset of $S^3$; the projection $p : N(S) \to S$ is a degree two covering map. If $a, b \in H_1(S)$ are represented respectively by embedded, oriented multicurves $\alpha, \beta \subset S$, then we define
$$
\mathcal{G}_S(a, b) = \text{lk}(\alpha, p^{-1}(\beta)),
$$
where lk is the linking number. If $S$ is a checkerboard surface, $\Gamma \subset \R^2$ its Tait graph, and $\mathcal{B}$ a 2-basis for $\Gamma$, we may orient the cycles in $\mathcal{B}$ to be compatible with an orientation of the plane. The oriented cycles $\mathcal{B}$ form a basis for $H_1(S;\Z)$, and the Goeritz matrix of $\mathcal{B}$ is a representation of $\mathcal{G}_S$ in this basis.

Let $L$ be an oriented link with diagram $D$ and checkerboard surface $S$. The \emph{oriented Euler number} of $S$ is defined to be $e(S,L) = -\mathcal{G}_S([L], [L])$, where $[L] \in H_1(S)$ is the homology class of $L$. If II$(D)$ is the set of all type II crossings of $D$, then Gordon and Litherland show that
$$
e(S,L) = -2\sum_{c \in \text{II}(D)} \sigma(c).
$$
Another direct computation shows $w(D) - w_0(D,S) = e(S,L) - \sum_{i \leq j} g_{ij}$, and subsequently we can recover the Jones polynomial using $e(S,L)$.
\begin{thm}
\label{thm:full_jones}
Let $L$ be an oriented link with checkerboard surface $S$ and associated Goeritz matrix $G$. Then
$$
J_L(t) = \big[ (-A)^{-3(e(S,L) - \sum_{i \leq j} g_{ij})} \mu[G] \big]_{t^{1/2} = A^{-2}}.
$$
\end{thm}

As stated in the introduction, Theorem \ref{thm:full_jones} says the Jones polynomial of a link $L$ can be explicitly computed from the Gordon-Litherland form of certain spanning surfaces of $L$, provided we choose the right basis.

\section{Invariants of Links in Thickened Surfaces}

\subsection{Polynomial Invariants}

When $G$ is a Goeritz matrix of a signed planar graph $\Gamma$, which is a Tait graph of a non-split link diagram $D$, we've seen that
$$
(-A)^{-3w_0(D,S)}\mu[G] = \tau[\Gamma] = \langle D \rangle.
$$
What happens when $\Gamma$ is not planar? The natural knot-theoretic objects to consider are links in thickened surfaces.

Every graph $\Gamma$ admits an embedding into some closed, orientable surface $\Sigma$. Given such an embedding, we can produce a link diagram $D \subset \Sigma$ using the \emph{medial construction}. To perform the medial construction, we take a regular neighborhood $S$ of $\Gamma$, which we may view as a subset of $\Sigma \times \{1/2\}$ in $\Sigma \times I$. We then insert a half-twist in $S$ around each edge $e$ of $\Gamma$, twisting in the direction indicated by the sign of $e$. Taking the boundary of $S$ and projecting down to $\Sigma \times \{0\}$ gives a diagram $D$ in $\Sigma$ whose Tait graph is $\Gamma$. Considering this construction, the following definitions and theorem give knot-theoretic significance to the polynomial $\tau$ for all signed graphs.

\begin{defn}
	Let $\Sigma$ be a closed, orientable surface, and $D \subset \Sigma$ a checkerboard-colorable, oriented link diagram. Let $S$ be a checkerboard surface of $D$ with associated Tait graph $\Gamma$. Viewing $\Gamma$ as a $1$-complex sitting in $S \subset \Sigma \times I$, let $dc(\Gamma)$ be the dimension of the cokernel of the inclusion-induced map $\iota : H_1(\Gamma) \to H_1(S)$. That is,
	$$
	dc(\Gamma) = \dim(H_1(S)/\iota(H_1(\Gamma)))
	$$
\end{defn}

The constant $dc(\Gamma)$ can be defined more combinatorially as follows: with notation as in the definition, let $\Gamma' \subset \Sigma$ be the other Tait graph of $D$, and define the {\em faces} of $\Gamma'$ be the components of $\Sigma - \Gamma'$. Then $dc(\Gamma)$ is the first Betti number of the union of the faces of $\Gamma'$. In particular, if $\Gamma'$ is {\em cellularly embedded} in $\Sigma$---meaning each face is a disk---then $dc(\Gamma) = 0$.

\begin{defn}
\label{def:nu}
Let $\Sigma$ be a closed, orientable surface, and $D \subset \Sigma$ a checkerboard-colorable, oriented link diagram. Let $\Gamma$ be an associated Tait graph, and define a polynomial $\nu_{D, \Gamma}$, in one variable $t$, by
$$
\nu_{D, \Gamma}(t) = [(-A^2 - A^{-2})^{dc(\Gamma)}(-A)^{-3w(D)} \tau[\Gamma]]_{t^{1/2} = A^{-2}}.
$$
\end{defn}

It can be proven by induction on the number of edges of $\Gamma$ that $\nu_{D, \Gamma}$ is an element of $\Z[t^{\pm1/2}]$. Equivalently, $(-A)^{-3w(D)} \tau[\Gamma]$ contains only even powers of $A$.

\begin{thm}
\label{thm:nu}
Let $\Sigma$ be a closed, orientable surface, $D \subset \Sigma$ a checkerboard-colorable link diagram, and $L \subset \Sigma \times I$ the associated link. Let $\Gamma$ and $\Gamma'$ be the Tait graphs associated to the two checkerboard surfaces of $D$. Then the set
$$
V = \{\nu_{D, \Gamma}(t), \nu_{D, \Gamma'}(t)\}
$$
is an isotopy invariant of $L$.
\end{thm}

\begin{proof}
The proof follows along the lines of the concluding remarks of \cite{k89}. Let $D_0 \subset \Sigma$ be another diagram of $L$, with Tait graphs $\Gamma_0, \Gamma_0'$. The diagram $D$ may transformed into $D_0$ by a set of diagram isotopies and local Reidemeister moves, shown in Figure \ref{fig:rm}. The results of these moves on the graphs $\Gamma$ and $\Gamma'$ are easy to determine---for example, Figure \ref{fig:rm_graph} shows the possible effects on the Tait graph of a Reidemeister II move. Figure \ref{fig:rmd} is the only case which changes the constant $dc$. Performing this set of moves will either transform $\Gamma$ into $\Gamma_0$ and $\Gamma'$ into $\Gamma'_0$, or will transform $\Gamma$ into $\Gamma'_0$ and $\Gamma'$ into $\Gamma_0$. Assume without loss of generality that the former holds; to prove the theorem, we check that $\nu_{D,\Gamma} = \nu_{D_0,\Gamma_0}$ and $\nu_{D,\Gamma'} = \nu_{D_0,\Gamma_0'}$. This is done by verifying directly that $\nu_{D,\Gamma}$ and $\nu_{D,\Gamma'}$ are unchanged by each Reidemeister move.
\end{proof}

\begin{figure}[H]
	\centering
	\includegraphics[height=2.5cm]{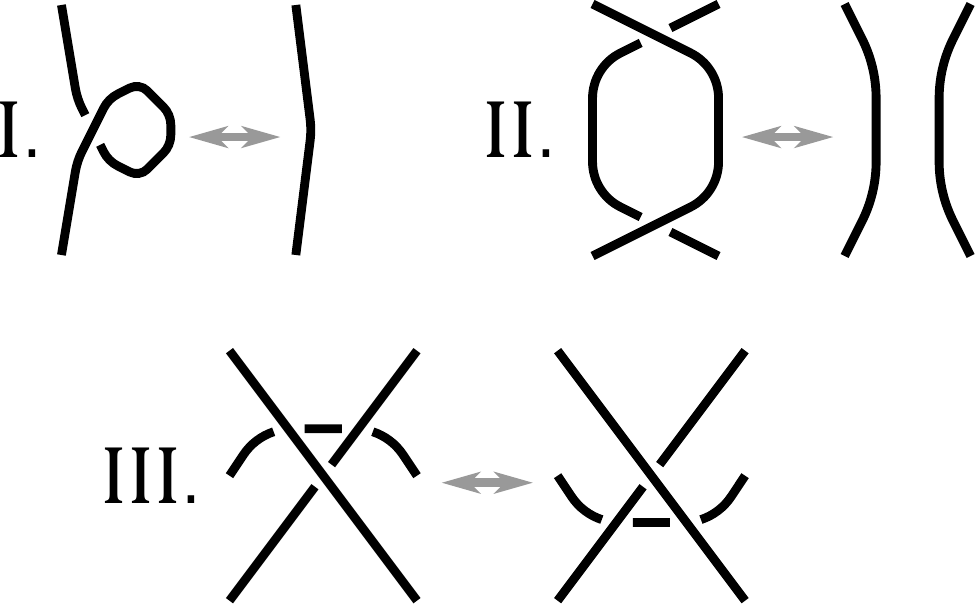}
	\caption{Reidemeister moves}
	\label{fig:rm}
\end{figure}

\begin{figure}[H]
	\centering
	\subcaptionbox{\label{fig:rma}}{
		\includegraphics[height=3cm]{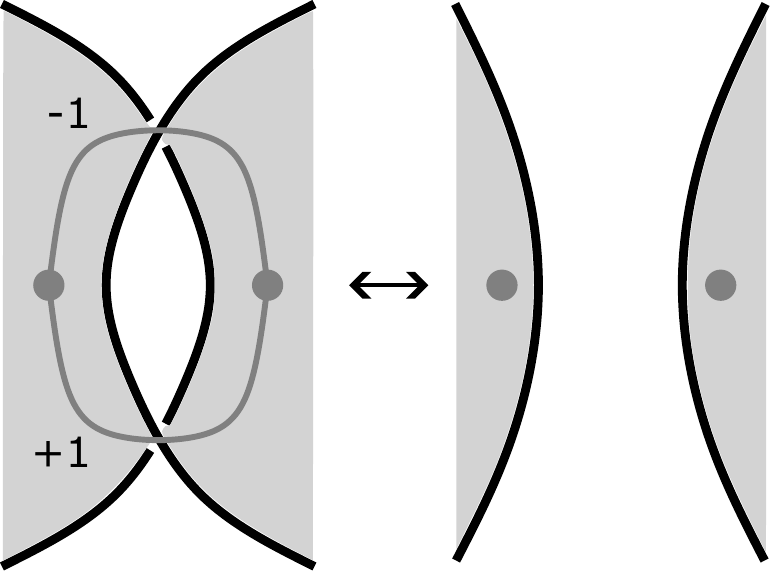}
	}
	\hspace{1cm}
	\subcaptionbox{\label{fig:rmb}}{
		\includegraphics[height=3cm]{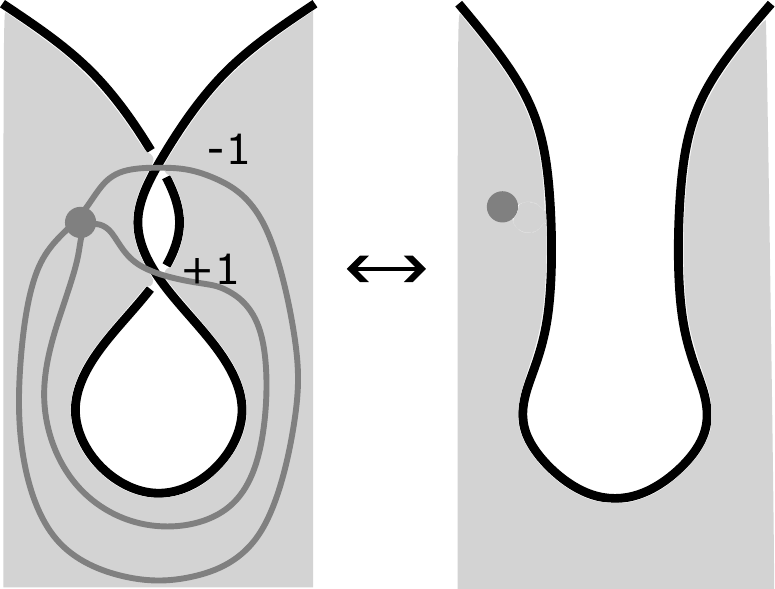}
	}
	\hspace{1cm}
	\subcaptionbox{\label{fig:rmc}}{
		\includegraphics[height=3cm]{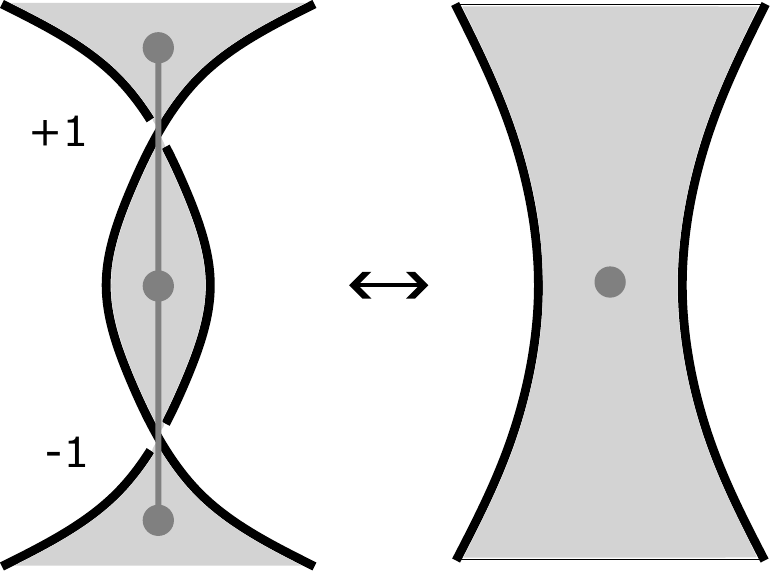}
	}
	\hspace{1cm}
	\subcaptionbox{\label{fig:rmd}}{
		\includegraphics[height=3cm]{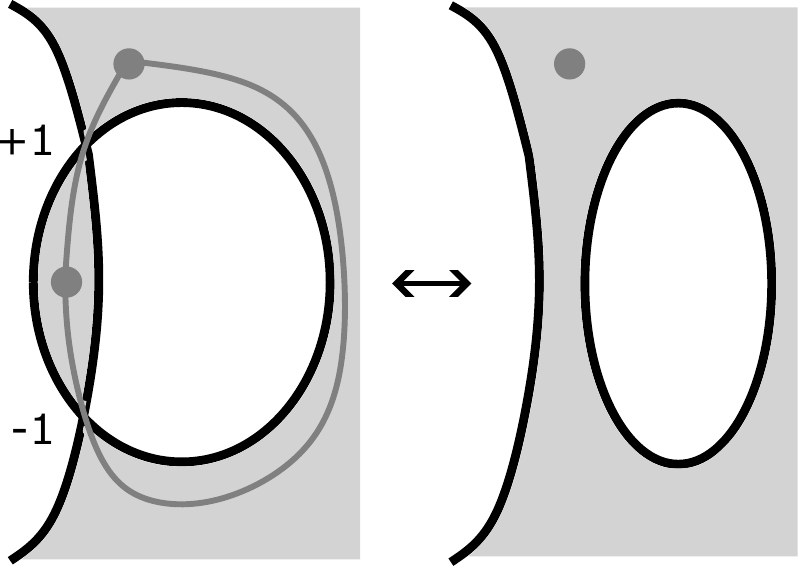}
	}
	\caption{Possible effects of an RII move}
	\label{fig:rm_graph}
\end{figure}

If $D$ is a classical link diagram in $\R^2$ (or $S^2$), then
$$
\nu_{D, \Gamma}(t) = \nu_{D, \Gamma'}(t) = J_L(t).
$$
For links in thickened surfaces of positive genus, however, the three polynomials are generally distinct. Consider, for example, the knot $K \subset T^2 \times I$ with diagram $D \subset T^2$ shown in Figure \ref{fig:vk}. We have:
\begin{align*}
\nu_{D, \Gamma}(t) &= 1 \\ 
\nu_{D, \Gamma'}(t) &= -t^{-1/2} - t^{-5/2} \\
J_K(t) &= t^{-1} + t^{-3} - t^{-4}.
\end{align*}

\begin{figure}[H]
\centering
\includegraphics[height=3.5cm]{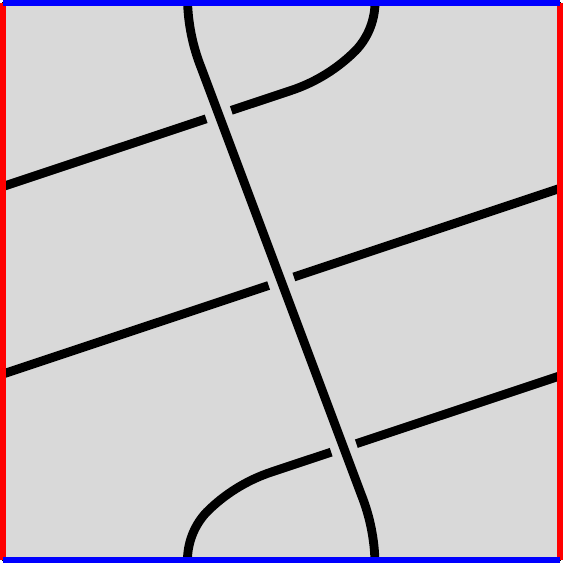}
\caption{A knot in the thickened torus}
\label{fig:vk}
\end{figure}

\begin{rmk}
\label{rmk:krushkal}
Suppose $D$ is a connected, alternating diagram, and $\Gamma$ and $\Gamma'$ are cellularly embedded in $\Sigma$. In this case, all three polynomials $\nu_{D, \Gamma}$, $\nu_{D, \Gamma'}$, and $J_L$ can be recovered from the \emph{Krushkal polynomial} of $\Gamma$. The Krushkal polynomial $P_\Gamma = P_{\Gamma, \Sigma} (X,Y,A,B)$ is a four-variable polynomial invariant of graphs embedded in closed, orientable surfaces \cite{k11}. $P$ satisfies a duality property: since $\Gamma$ and $\Gamma'$ are cellular and dual with respect to their embeddings in $\Sigma$,
$$
P_\Gamma(X,Y,A,B) = P_{\Gamma'}(Y,X,B,A).
$$
The Krushkal polynomial also specializes to the two-variable Tutte polynomial $T_\Gamma(X,Y)$. When $D$ is alternating, $\tau[\Gamma]$ can be recovered from $T_\Gamma$ \cite{t87} and $J_L$ can be recovered from $P$. Additionally, Krushkal notes that one can define a \emph{signed} version of $P$, analogous to Kauffman's generalization of $\tau$ \cite{k89}. From such a polynomial it may be possible to recover $J_L$, $\nu_{D, \Gamma}$, and $\nu_{D, \Gamma'}$ even when $D$ is not alternating. We emphasize, however, that $P$ itself is not a link (or matroid) invariant.
\end{rmk}

In the next section, we discuss how the polynomial $\nu$ relates to the determinants of links in thickened surfaces.

\subsection{Goeritz Matrices and Determinants of Links in Surfaces}

Throughout this section, let $\Sigma$ be a closed, orientable surface, and $D \subset \Sigma$ a checkerboard-colorable diagram of a non-split link $L \subset \Sigma \times I$. Let $S$ and $S'$ be the two checkerboard surfaces of $D$, and $\Gamma$ and $\Gamma'$ their respective Tait graphs.

Let $G$ and $G'$ be Goeritz matrices of $S$ and $S'$ respectively, according to Definition \ref{def:ill}. If $\Sigma = S^2$, $G$ is also a Goeritz matrix of $M(\Gamma)$ by our Definition \ref{def:goeritz}. Further, since $M(\Gamma) = B(\Gamma')$, $G$ is trivially a Goeritz matrix of $B(\Gamma')$.

If $\Sigma$ has positive genus, the situation is more complicated. Since the embedding $\Gamma \hookrightarrow \Sigma$ is not planar, it is no longer true in general that $M(\Gamma) = B(\Gamma')$. Additionally, $G$ is not necessarily a Goeritz matrix for $M(\Gamma)$ by Definition \ref{def:goeritz}. Indeed, if $\Gamma$ is not a planar graph, then $M(\Gamma)$ is not cographic and no Goeritz matrix of $M(\Gamma)$ can be defined. However, it remains true that $G$ is a Goeritz matrix for the bond matroid $B(\Gamma')$.

\begin{prop}
\label{thm:goeritz_eq}
If $G$ is a Goeritz matrix for $S$ by Definition \ref{def:ill}, then $G$ is a Goeritz matrix for $B(\Gamma')$ by Definition \ref{def:goeritz}. Dually, if $G$ is a Goeritz matrix for $S'$, then $G$ is a Goeritz matrix for $B(\Gamma)$.
\end{prop}

\begin{proof}
This is clear from the definition of each kind of Goeritz matrix, noting that the faces $X_1, \dots, X_m$ of $\Sigma - S$ in the construction of $G$ in Definition \ref{def:ill} correspond exactly to elements of a 2-basis of $B(\Gamma')$. Our sign convention for dual matroids ensures each point of $B(\Gamma')$ has the same sign as the edge of $\Gamma$ it intersects.
\end{proof}

Proposition \ref{thm:goeritz_eq} has two applications. First, it shows that every symmetric, integer matrix is the Goeritz matrix of some link in a thickened surface. This result extends \cite[Thm.~3.5]{bck21}.

\begin{thm}
\label{thm:every_matrix}
Every symmetric, integer matrix $G$ is a Goeritz matrix of a checkerboard-colorable link diagram in a thickened surface.
\end{thm}

\begin{proof}
Let $\Gamma$ be the signed graph constructed in the proof of Corollary \ref{thm:all_signed}, whose bond matroid has $G$ as a Goeritz matrix, and let $\Sigma$ be a closed, orientable surface into which $\Gamma$ embeds. By performing the medial construction on $\Gamma$, we construct a link diagram $D$ which has $G$ as a Goeritz matrix.
\end{proof}

The second application of Proposition \ref{thm:goeritz_eq} relates the polynomials $\nu_{D, \Gamma}$ and $\nu_{D, \Gamma'}$ to the determinants of $L$, discussed in Section 2.2.

\begin{thm}
\label{thm:dets}
Let $\Gamma$, $\Gamma'$, $G$, $G'$, and $L$ be as defined above, with $L$ oriented. Then
\begin{align*}
|\nu_{D, \Gamma}(-1)| = |\det(G')| \\
|\nu_{D, \Gamma'}(-1))| = |\det(G)|.
\end{align*}
In particular, $\det{L} = \{|\nu_{D, \Gamma}(-1)|, |\nu_{D, \Gamma'}(-1)|\}$.
\end{thm}

Theorem \ref{thm:dets} generalizes the classical case, where $\det(L) = |J_L(-1)|$. It may be surprising that the polynomial $\nu_{D, \Gamma}$ corresponds to the matrix $G'$ rather than the matrix $G$---an analogous phenomenon, called \emph{chromatic duality}, was observed in \cite{bck21} when comparing the determinant and signature defined there with those of \cite{ill10}.
\begin{lemma}
\label{thm:det}
Let $G = (g_{k\ell})$ be an $m$-by-$m$ symmetric, integer matrix, and fix $\zeta = e^{i\pi/4}$. Then
\begin{equation}
\label{eq:det}
\mu[G] |_{A = \zeta} = \zeta^{(3m + \sum_{i \leq j} g_{ij})} \det G.
\end{equation}
In particular, $\big|\mu[G] |_{A = \zeta} \big| = \big| \det G \big|$.
\end{lemma}

Assuming this lemma, we have:

\begin{proof}[Proof of Theorem \protect \ref{thm:dets}]
By Proposition \ref{thm:goeritz_eq}, $G'$ is a Goeritz matrix for $B(\Gamma)$. Using Lemma \ref{thm:det} and Remark \ref{rmk:duality},
$$
|\nu_{D, \Gamma}(-1)| = |\tau[\Gamma](e^{i\pi/4})| = |\tau[B(\Gamma)](e^{i\pi/4})|  = |\mu[G'](e^{i\pi/4})| = |\det(G')|.
$$
The proof for $\Gamma'$ and $G$ is the same.
\end{proof}

We now prove the lemma.

\begin{proof}[Proof of Lemma \protect \ref{thm:det}]
First, we calculate
$$
P_n(\zeta) = n\zeta^{n + 3},
$$
where $P_n$ is the polynomial of Definition \ref{def:polys}. Additionally, since $(-A^{-2} - A^2)|_{A = \zeta} = 0$, relations (ii) and (iii) of Definition \ref{def:new_bracket} reduce to:
\begin{enumerate}[label=(\roman*)]
\addtocounter{enumi}{1}
\item $\mu \begin{bmatrix} g_{ii} & r_i & g_{ij} \\
c_i & * & c_j \\
g_{ij} & r_j & g_{jj} \\
\end{bmatrix} = 
\zeta^{-g_{ij}} \mu \begin{bmatrix}
g_{ii} + g_{ij} & r_i & 0 \\
c_i & * & c_j \\
0 & r_j & g_{jj} + g_{ij}
\end{bmatrix}  -
g_{ij}\zeta^{3-g_{ij}} \mu \begin{bmatrix}
g_{ii} + g_{jj} + 2g_{ij}  & r_i + r_j\\
c_i + c_j & *
\end{bmatrix} $
\item $\mu \begin{bmatrix}
g_{i - 1, i - 1} & 0 & \dots\\
0 & g_{ii} & 0 \\
\dots & 0 & g_{i + 1, i + 1}
\end{bmatrix} =
g_{ii}\zeta^{g_{ii} + 3}
\mu \begin{bmatrix}
g_{i - 1, i - 1} & \dots \\
\dots & g_{i + 1, i + 1}
\end{bmatrix} $
\end{enumerate}
In each relation above we've shown only part of each matrix. The variable $r_i$ in (ii) indicates the portion of row $i$ from $g_{i, i + 1}$ to $g_{i, j - 1}$, and $c_i$ is the portion of column $i$ from $g_{i + 1, i}$ to $g_{j - 1, i}$. The symbol $*$ is the square block bound by $c_i$, $r_i$, $c_j$ and $r_j$.

Using (iii), equation (\ref{eq:det}) is easy to verify if $G$ is diagonal. The proof proceeds by induction on the number of nonzero, off-diagonal elements of $G$. Let $\alpha = \sum_{i \leq j} g_{ij}$, and let $g_{ij}$ (for fixed $i$ and $j$) be a nonzero, off-diagonal element of $G$. With notation as above, (ii) and the induction hypothesis give
\begin{align*}
\mu \begin{bmatrix} g_{ii} & r_i & g_{ij} \\
c_i & * & c_j \\
g_{ij} & r_j & g_{jj} \\
\end{bmatrix} &= 
\zeta^{3m + \alpha} \det \begin{bmatrix}
g_{ii} + g_{ij} & r_i & 0 \\
c_i & * & c_j \\
0 & r_j & g_{jj} + g_{ij}
\end{bmatrix}  -g_{ij}\zeta^{3m + \alpha} \det \begin{bmatrix}
g_{ii} + g_{jj} + 2g_{ij}  & r_i + r_j\\
c_i + c_j & *
\end{bmatrix}.
\end{align*}
The proof is finished by showing
$$
\det \begin{bmatrix} g_{ii} & r_i & g_{ij} \\
c_i & * & c_j \\
g_{ij} & r_j & g_{jj} \\
\end{bmatrix} = \det \begin{bmatrix}
g_{ii} + g_{ij} & r_i & 0 \\
c_i & * & c_j \\
0 & r_j & g_{jj} + g_{ij}
\end{bmatrix} \\
- g_{ij} \det \begin{bmatrix}
g_{ii} + g_{jj} + 2g_{ij}  & r_i + r_j\\
c_i + c_j & *
\end{bmatrix}.
$$
Letting $G$ denote the matrix on the left side of the equation above, we calculate:
\begin{align*}
&\det \begin{bmatrix}
g_{ii} + g_{ij} & r_i & 0 \\
c_i & * & c_j \\
0 & r_j & g_{jj} + g_{ij}
\end{bmatrix} \\
&= 
\det G + g_{ij} \Big(
\det \begin{bmatrix}
1 & r_i & g_{ij} \\
0 & * & c_j \\
-1 & r_j & g_{jj}
\end{bmatrix} +  
\det \begin{bmatrix}
g_{ii} & r_i & -1 \\
c_i & * & 0 \\
g_{ij} & r_j & 1
\end{bmatrix}
+ \det \begin{bmatrix}
1 & r_i & -g_{ij} \\
0 & * & 0 \\
-1 & r_j & g_{ij}
\end{bmatrix} \Big) \\
&= \det G + g_{ij} \Big( 
\det \begin{bmatrix}
* & c_j \\
r_j & g_{jj}
\end{bmatrix} + (-1)^{i + j + 1}
\det \begin{bmatrix}
 r_i & g_{ij} \\
* & c_j 
\end{bmatrix} + 
\det \begin{bmatrix}
g_{ii} & r_i \\
c_i & *
\end{bmatrix} \\
& \ \ \ \ \ +
(-1)^{i + j + 1}
\det \begin{bmatrix}
c_i & * \\
g_{ij} & r_j
\end{bmatrix} + 
\det \begin{bmatrix}
* & 0 \\
 r_j & g_{ij}
\end{bmatrix} + (-1)^{i + j + 1}
\det \begin{bmatrix}
r_i & -g_{ij} \\
* & 0
\end{bmatrix} \Big) \\
&= \det G + g_{ij}
\det \begin{bmatrix}
g_{ii} + g_{jj} + 2g_{ij}  & r_i + r_j\\
c_i + c_j & *
\end{bmatrix}.
\end{align*}
This completes the proof.
\end{proof}

We conclude with two remarks.

\begin{rmk}
Let $D \subset \Sigma$ be a non-split, checkerboard-colorable, alternating link diagram, such that the Tait graphs $\Gamma$ and $\Gamma'$ of $D$ are cellularly embedded. Let $G$ and $G'$ be respective Goeritz matrices, and suppose $\Gamma$ has positive-signed edges. Carrying out the specializations mentioned in Remark \ref{rmk:krushkal}, we have:
\begin{align*}
|\det(G')| &= \lim_{t \to -1} |(-t^{-1} - 1)^gP_{\Gamma, \Sigma}(-t - 1, -t^{-1} - 1, -t^{-1} - 1, (-t^{-1} - 1)^{-1})| \\
|\det(G)| &= \lim_{t \to -1}  |(-t - 1)^gP_{\Gamma, \Sigma}(-t - 1, -t^{-1} - 1, (-t - 1)^{-1}, -t - 1)|,
\end{align*}
where $P$ is the Krushkal polynomial.
\end{rmk}

\begin{rmk}
The results of Section 6 can be placed in the context of \emph{virtual links}, which are equivalence classes of links in thickened surfaces under natural stabilization and destabilization operations. Every virtual link is uniquely represented, up to diffeomorphism, by a link in a minimal genus thickened surface \cite{k03}, and in this way every diffeomorphism invariant of links in thickened surfaces produces an invariant of virtual links. In our case, the two polynomials of $V$ in Theorem \ref{thm:nu} give invariants of checkerboard-colorable virtual links which satisfy the same determinant properties.
\end{rmk}

\bibliography{volume_conjecture}{}

\bibliographystyle{amsplain}
\end{document}